\documentclass[a4paper]{amsart}
\usepackage[english, american]{babel}
\usepackage[latin1]{inputenc}
\usepackage{amssymb}
\usepackage{amsmath}
\usepackage{mathrsfs}
\usepackage{eufrak}
\usepackage{amsthm}
\usepackage{amsfonts}
\usepackage{textcomp}
\usepackage{graphicx}
\usepackage[pdftex]{color}
\usepackage{paralist}
\usepackage{hyperref}
\usepackage{comment}
\usepackage[arrow, matrix, curve]{xy}

\swapnumbers
\newtheoremstyle{dotless}{}{}{}{}{\bfseries}{}{\newline}{}
\theoremstyle{dotless}
\newtheorem{Definition}{Definition}
\numberwithin{Definition}{section}
\newtheorem{Minimal projective resolutions and minimal injective resolutions in Nakayama algebras}[Definition]{Minimal projective resolutions and minimal injective resolutions in Nakayama algebras}
\newtheorem{Length of the indecomposable left module}[Definition]{Length of the indecomposable left module}

\newtheorem{Corollary}[Definition]{Corollary}
\newtheorem{Lemma}[Definition]{Lemma}

\newtheorem{Structure of indecomposable injective modules}[Definition]{Structure of indecomposable injective modules}
\newtheorem{Calculation of dominant dimension for a Nakayama algebra}[Definition]{Calculation of dominant dimension for a Nakayama algebra}
\newtheorem{Theorem of Mueller}[Definition]{Theorem of Mueller}
\newtheorem{Conjecture about Delta A}[Definition]{Conjecture about $\Delta_A$}
\newtheorem{Two conjectures}[Definition]{Two conjectures}
\newtheorem{lemma}[Definition]{Lemma}
\newtheorem{theorem}[Definition]{Theorem}
\newtheorem*{Result 1}{Result 1}
\newtheorem*{Result 2}{Result 2}
\newtheorem*{Theorem A}{Theorem A}
\newtheorem*{Theorem B}{Theorem B}

\newtheorem{example}[Definition]{Example}
\newtheorem{proposition}[Definition]{Proposition}
\newtheorem{Proposition}[Definition]{Proposition}

\newtheorem{Weak Gorenstein projective conjecture(WGPC)}[Definition]{Weak Gorenstein projective conjecture(WGPC)}
\newtheorem{Gorenstein projective conjecture(GPC)}[Definition]{Gorenstein projective conjecture(GPC)}

\newtheorem*{Proof*}{Proof}

\newtheorem{examples}[Definition]{Examples}
\newtheorem{remark}[Definition]{Remark}
\makeatletter
\renewcommand*\env@matrix[1][\
arraystretch]{%
  \edef\arraystretch{#1}%
  \hskip -\arraycolsep
  \let\@ifnextchar\new@ifnextchar
  \array{*\c@MaxMatrixCols c}}
\makeatother

\begin{document}

\title{A bocs theoretic characterization of gendo-symmetric algebras}

\author{Ren\'{e} Marczinzik}

\subjclass[2010]{Primary 16G10, 16E10}

\keywords{Representation theory of finite dimensional algebras, corings, dominant dimension}

\begin{abstract}
Gendo-symmetric algebras were recently introduced by Fang and K\"onig in \cite{FanKoe}. An algebra is called gendo-symmetric in case it is isomorphic to the endomorphism ring of a generator over a finite dimensional symmetric algebra.
We show that a finite dimensional algebra $A$ over a field $K$ is gendo-symmetric if and only if there is a bocs-structure on $(A,D(A))$, where $D=Hom_K(-,K)$ is the natural duality. Assuming that $A$ is gendo-symmetric, we show that the module category of the bocs $(A,D(A))$ is isomorphic to the module category of the algebra $eAe$, when $e$ is an idempotent such that $eA$ is the unique minimal faithful projective-injective right $A$-module. We also prove some new results about gendo-symmetric algebras using the theory of bocses.
\end{abstract}

\maketitle

\section*{Introduction}
A bocs is a generalization of the notion of coalgebra over a field. Bocses are also known under the name coring (see the book \cite{BreWis}). A famous application of bocses has been the proof of the tame and wild dichtomy theorem by Drozd for finite dimensional algebras over an algebraically closed field (see \cite{Dro} and the book \cite{BSZ}). For any given bocs $(A,W)$ over a finite dimensional algebra, one can define a corresponding module category and analyze it. Given a finite dimensional algebra $A$ over a field $K$, it is an interesting question whether for a given $A$-bimodule $W$, there exists a bocs structure on $(A,W)$. The easiest example to consider is the case $W=A$ and in this case the module category one gets is just the module category of the algebra $A$. Every finite dimensional algebra has a duality $D=Hom_K(-,K)$ and so the next example of an $A$-bimodule to consider is perhaps $W=D(A)$. We will characterize all finite dimensional algebras $A$ such that there is a bocs structure on $(A,D(A))$ and find a surprising connection to a recently introduced class of algebras generalizing symmetric algebras (see \cite{FanKoe2}). Those algebras are called gendo-symmetric and are defined as endomorphism rings of generators of symmetric algebras. Alternatively these are the algebras $A$, where there exists an idempotent $e$ such that $eA$ is a minimal faithful injective-projective module and $D(Ae) \cong eA$ as $(eAe,A)$-bimodules. Then $eAe$ is the symmetric algebra such that $A \cong End_{eAe}(M)$, for an $eAe$-module $M$ that is a generator of mod-$eAe$. Famous examples of non-symmetric gendo-symmetric algebras are Schur algebras $S(n,r)$ with $n \geq r$ and blocks of the Bernstein-Gelfand-Gelfand category $\mathcal{O}$ of a complex semisimple Lie algebra (for a proof of this, using methods close to ours, see \cite{KSX} and for applications see \cite{FanKoe3}). The first section provides the necessary background on bocses and algebras with dominant dimension larger or equal 2.
The second section proves our main theorem: 
\newpage
\begin{Theorem A} 
(Theorem \hyperref[mainresult]{ \ref*{mainresult}}) \newline
A finite dimensional algebra $A$ is gendo-symmetric if and only if $(A,D(A))$ has a bocs-structure.
\end{Theorem A}

We also provide some new structural results about gendo-symmetric algebras in this section. For example we show, using bocs-theoretic methods, that the tensor product over the field $K$ of two gendo-symmetric algebras is again gendo-symmetric and we proof that $Hom_{A^{e}}(D(A),A)$ is isomorphic to the center of $A$, where $A^{e}$ denotes the enveloping algebra of $A$. \newline
In the final section, we describe the module category $\mathcal{B}$ of the bocs $(A,D(A))$ in case $A$ is gendo-symmetric.
The following is our second main result:
\begin{Theorem B}
(Theorem \hyperref[mainresult2]{ \ref*{mainresult2}}) \newline
Let $A$ be a gendo-symmetric algebra with minimal faithful projective-injective module $eA$. Then the module category of the bocs $(A,D(A))$ is equivalent to $eAe$-mod as $K$-linear categories.
\end{Theorem B}

I thank Steffen K\"onig for useful comments and proofreading. I thank Julian K\"ulshammer for providing me with an early copy of his article \cite{Kue}.

\section{Preliminaries}
We collect here all needed definitions and lemmas to prove the main theorems.
Let an algebra always be a finite dimensional algebra over a field $K$ and a module over such an algebra is always a finite dimensional right module, unless otherwise stated. $D=Hom_A(-,K)$ denotes the duality for a given finite dimensional algebra $A$. $mod-A$ denotes the category of finite dimensional right $A$-modules and $proj$ ($inj$) denotes the subcategory of finitely generated projective (injective) $A$-modules.
We note that we often omit the index in a tensor product, when we calculate with elements. We often identify $A \otimes_A X \cong X$ for an $A$-module $X$ without explicitly mentioning the natural isomorphism.
The Nakayama functor $\nu : mod-A \rightarrow mod-A$ is defined as $DHom_A(-,A)$ and is isomorphic to the functor $(-) \otimes_A D(A)$. The inverse Nakayama functor $\nu^{-1} : mod-A \rightarrow mod-A$ is defined as $Hom_{A^{op}}(-,A)D$ and is isomorphic to the functor $Hom_A(D(A),-)$ (see \cite{SkoYam} Chapter III section 5 for details). The Nakayama functors play a prominent role in the representation theory of finite dimensional algebras, since $\nu: proj \rightarrow inj$ is an equivalence with inverse $\nu^{-1}$. For example they appear in the definition of the Auslander-Reiten translates $\tau$ and $\tau^{-1}$ (see \cite{SkoYam} Chapter III. for the definitions):
\begin{proposition}
\label{nakafunc}
Let $M$ be an $A$-module with a minimal injective presentation $0 \rightarrow M \rightarrow I_0 \rightarrow I_1$. Then the following sequence is exact: \newline
$0 \rightarrow \nu^{-1}(M) \rightarrow \nu^{-1}(I_0) \rightarrow \nu^{-1}(I_1) \rightarrow \tau^{-1}(M) \rightarrow 0$.
\end{proposition}
\begin{proof}
See \cite{SkoYam}, Chapter III. Proposition 5.3. (ii).
\end{proof}
\noindent The \textit{dominant dimension} domdim($M$) of a module $M$ with a minimal injective resolution $(I_i): 0 \rightarrow M \rightarrow I_0 \rightarrow I_1 \rightarrow ...$  is defined as: \newline
domdim($M$):=$\sup \{ n | I_i $ is projective for $i=0,1,...,n \}$+1, if $I_0$ is projective, and \newline domdim($M$):=0, if $I_0$ is not projective. \newline
The dominant dimension of a finite dimensional algebra is defined as the dominant dimension of the regular module $A_A$.
It is well-known that an algebra $A$ has dominant dimension larger than or equal to 1 iff there is an idempotent $e$ such that $eA$ is a minimal faithful projective-injective module. The Morita-Tachikawa correspondence (see \cite{Ta} for details) says that the algebras, which are endomorphism rings of generator-cogenerators are exactly the algebras with dominant dimension at least 2. The full subcategory of modules of dominant dimension at least $i \geq 1$ is denoted by $Dom_i$.
$A$ is called a \textit{Morita algebra} iff it has dominant dimension larger than or equal to 2 and $D(Ae) \cong eA$ as $A$-right modules. This is equivalent to $A$ being isomorphic to $End_B(M)$, where $B$ is a selfinjective algebra and $M$ a generator of mod-$B$ (see \cite{KerYam}).
$A$ is called a \textit{gendo-symmetric algebra} iff it has dominant dimension larger than or equal to 2 and $D(Ae) \cong eA$ as $(eAe,A)-$bimodules iff it has dominant dimension larger than or equal to 2 and $D(eA) \cong Ae$ as $(A,eAe)$-bimodules. This is equivalent to $A$ being isomorphic to $End_B(M)$, where $B$ is a symmetric algebra and $M$ a generator of mod-$B$ and in this case $B=eAe$ (see \cite{FanKoe}).

\begin{Proposition}
\label{domdim2}
Let $A$ be a gendo-symmetric algebra and $M$ an $A$-module. Then $M$ has dominant dimension larger or equal to two iff  $\nu^{-1}(M) \cong M$.
\end{Proposition}
\begin{proof}
See \cite{FanKoe2}, proposition 3.3.
\end{proof}
The following result gives a formula for the dominant dimension of Morita algebras:
\begin{Proposition}
\label{APT}
Let $A$ be a Morita algebra with minimal faithful projective-injective module $eA$ and $M$ an $A$-module. Then $domdim(M)= \inf \{ i \geq 0 | Ext^{i}(A/AeA,M) \neq 0 \}$. Especially, $Hom_A(A/AeA,A)=0$ for every Morita algebra, since they always have dominant dimension at least 2.
\begin{proof}
This is a special case of \cite{APT}, Proposition 2.6.
\end{proof}
\end{Proposition}
The following lemma gives another characterization of gendo-symmetric algebras, which is used in the proof of the main theorem. 
\begin{lemma}
\label{anotherchara}
Let $A$ be a finite dimensional algebra. Then $A$ is a gendo-symmetric algebra iff $D(A) \otimes_A D(A) \cong D(A)$ as $A$-bimodules. Assume $eA$ is the minimal faithful projective-injective module. In case $A$ is gendo-symmetric, $D(A) \cong Ae \otimes_{eAe} eA$ as $A$-bimodules.
\end{lemma}
\begin{proof}
See \cite{FanKoe2} Theorem 3.2. and \cite{FanKoe} in the construction of the comultiplication following Definition 2.3.
\end{proof}
\begin{lemma}
An $A$-module $P$ is projective iff there are elements $p_1,p_2,...,p_n \in P$ and elements $\pi_1,\pi_2,...,\pi_n \in Hom_A(P,A)$ such that the following condition holds: \newline
$x=\sum\limits_{i=1}^{n}{ p_i \pi_i(x)}$ for every $x \in P$. \newline
We then call the $p_1,...,p_n$ a \textit{probasis} and $\pi_1,...,\pi_n$ a \textit{dual probasis} of $P$.
\end{lemma}
\begin{proof}
See \cite{Rot} Propostion 3.10. 
\end{proof}
\begin{example}
Let $P=eA$, for an idempotent $e$. Then a probasis is given by $p_1=e$ and the dual probasis is given by $\pi_1=l_e \in Hom_A(eA,A)$, which is left multiplication by $e$. $l_e$ can be identified with $e$ under the $(A,eAe)$-bimodule isomorphism $Ae \cong Hom_A(eA,A)$.
\end{example}
\begin{proposition}
\label{faithfulhom}
1. $Hom_A(D(A),A)$ is a faithful right $A$-module iff there is an idempotent $e$, such that $eA$ and $Ae$ are faithful and injective. \newline
2. Let $A$ be an algebra with $Hom_A(D(A),A) \cong A$ as right $A$-modules, then $A$ is a Morita algebra.
\end{proposition}
\begin{proof}
1. See \cite{KerYam}, Theorem 1. \newline
2. See \cite{KerYam}, Theorem 3. 
\end{proof}

\begin{lemma}
\label{dualiso}
Let $Y$ and $Z$ be $A$-bimodules. Then the following is an isomorphism of $A$-bimodules: \newline
$$Hom_A(Y,D(Z)) \cong D(Y \otimes_A Z).$$
\end{lemma}
\begin{proof}
See \cite{ASS} Appendix 4, Proposition 4.11. 
\end{proof}

\begin{Definition}
Let $A$ be a finite dimensional algebra and $W$ an $A$-bimodule and let $c_r: W \rightarrow A \otimes_A W $ and $c_l: W \rightarrow W \otimes_A A$ be the canonical isomorphisms. Then the tuple $\mathcal{B}:=(A,W)$ is called a \textit{bocs} (see \cite{Kue}) or the module $W$ is called an $A$-coring (see \cite{BreWis}) if there are $A$-bimodule maps $\mu: W \rightarrow W \otimes_A W$ (the comultiplication) and $\epsilon: W \rightarrow A$ (the counit) with the following properties: \newline 
$  (1_W \otimes_A \epsilon)\mu=c_l, (\epsilon \otimes_A 1_W)\mu=c_r$ and $(\mu\otimes_A 1_W)\mu=(1_W \otimes_A \mu)\mu$.
We often say for short that $W$ is a bocs, if $A$ (and $\mu$ and $\epsilon$) are clear from the context.
The category of the finite dimensional bocs modules is defined as follows: \newline
Objects are the finite dimensional right $A$-modules. \newline
Homomorphism spaces are $Hom_\mathcal{B}(M,N):=Hom_A(M,Hom_A(W,N))$ with the following composition $*$ and units: \newline
Let $g: M \rightarrow Hom_A(W,N) \in Hom_\mathcal{B}(M,N)$ and $f: L \rightarrow Hom_A(W,M) \in Hom_\mathcal{B}(L,M)$. Then $g*f:= Hom_A(\mu,N) \psi Hom_A(W,g)f $, where $\psi$ is the adjunction isomorphism $Hom_A(W,Hom_A(W,N)) \rightarrow Hom_A(W \otimes_A W,N)$. The units $1_M \in Hom_\mathcal{B}(M,M)$ are defined as follows: $1_M := Hom_A(\epsilon,M) \xi$, where $\xi: M \rightarrow Hom_A(A,M)$ is the canonical isomorphism.
Note that the module category of a bocs is $K$-linear. We refer to \cite{Kue} for other equivalent descriptions of the bocs module category and more information.
\end{Definition}
\begin{examples}
\label{probasis} 
1. $(A,A)$ is always a bocs with the obvious multiplication and comultiplication. The next natural bimodule to look for a bocs-structure is $D(A)$. We will see that $(A,D(A))$ is not a bocs for arbitrary finite dimensional algebras. \newline
2. The next example can be found in 17.6. in \cite{BreWis}, to which we refer for more details. Let $P$ be a $(B,A)$-bimodule for two finite dimensional algebras $B$ and $A$ such that $P$ is projective as a right $A$-module and let $P^{*}:=Hom(P,A)$, which is then a $(A,B)$ bimodule. Let $p_1,p_2,...,p_n$ be a probasis for $P$ and $\pi_1,\pi_2,...,\pi_n$ a dual probasis of the projective $A$-module $P$.
Denote the $A$-bimodule $P^{*} \otimes_B P$ by $W$ and define the comultiplication $\mu: W \rightarrow W \otimes_A W$ as follows: Let $f \in P^{*}$ and $p \in P$, then $\mu(f \otimes p)=\sum\limits_{i=1}^{n}{(f \otimes p_i) \otimes (\pi_i \otimes p)}$. Define the counit $\epsilon : W \rightarrow A$ as follows: $\epsilon( f \otimes p)=f(p)$. Now specialise to $P=eA$, for an idempotent $e$ and identify $Hom_A(eA,A)=Ae$. Then $\mu(ae \otimes eb) = (ae \otimes e) \otimes (e \otimes eb)$ and $\epsilon(ae \otimes eb)=aeb$. We will use this special case in the next section to show that $(A,D(A))$ is always a bocs for a gendo-symmetric algebra. \newline
3. Let $(A_1,W_1)$ and $(A_2,W_2)$ be bocses, then $(A_1 \otimes_K A_2 , W_1 \otimes_K W_2)$ is again a bocs. See \cite{BreWis} 24.1. for a proof.
\end{examples}
\section{Characterization of gendo-symmetric algebras}
\setcounter{subsection}{1}
\setcounter{Definition}{0}
The following lemma, will be important for proving the main theorem.
\begin{lemma}
\label{domdim1}
Assume that $Hom_A(D(A),A) \cong A \oplus X$ as right $A$-modules for some right $A$-module $X$, then $domdim(A) \geq 2$ and $X=0$.
\end{lemma}
\begin{proof}
By assumption $Hom_A(D(A),A)$ is faithful and so there is an idempotent $e$ with $eA$ and $Ae$ faithful and injective by \hyperref[faithfulhom]{ \ref*{faithfulhom}} 1., which implies that $A$ has dominant dimension at least 1. Choose $e$ minimal such that those properties hold. Now look at the minimal injective presentation $0 \rightarrow A \rightarrow I_0 \rightarrow I_1$ of $A$ and note that $I_0 \in add(eA)$. Using \hyperref[nakafunc]{ \ref*{nakafunc}}, there is the following exact sequence:
$0 \rightarrow \nu^{-1}(A) \rightarrow \nu^{-1}(I_0) \rightarrow \nu^{-1}(I_1) \rightarrow \tau^{-1}(A) \rightarrow 0$. But $\nu^{-1}(A) \cong Hom_A(D(A),A) \cong A \oplus X$ and so there is the embedding:
$0 \rightarrow A \oplus X \rightarrow \nu^{-1}(I_0)$. Note that $\nu^{-1}(I_0) \in add(eA)$ is the injective hull of $A \oplus X$, since $\nu^{-1}: inj \rightarrow proj$ is an equivalence and $eA$ is the minimal faithful projective injective module. Thus $\nu^{-1}(I_0)$ has the same number of indecomposable direct summands as $I_0$. Therefore $soc(X)=0$ and so $X=0$, since every indecomposable summand of the socle of the module provides an indecomposable direct summand of the injective hull of that module.
Thus $Hom_A(D(A),A) \cong A$ and $A$ is a Morita algebra by  \hyperref[faithfulhom]{ \ref*{faithfulhom}} 2. and so $A$ has dominant dimension at least 2.
\end{proof}

We now give a bocs-theoretic characterization of gendo-symmetric algebras.
\begin{theorem}
\label{mainresult}
Let $A$ be a finite dimensional algebra.
Then the following are equivalent: \newline
1. $A$ is gendo-symmetric. \newline
2. There is a comultiplication and counit such that $\mathcal{B}=(A,D(A))$ is a bocs.
\end{theorem}
\begin{proof}
We first show that 1. implies 2.: \newline
Assume that $A$ is gendo-symmetric with minimal faithful projective-injective module $eA$. Set $P:=eA$ and apply the second example in \hyperref[probasis]{ \ref*{probasis}}, with $B:=eAe$, to see that $\mathcal{B}:=(A,Ae \otimes_{eAe} eA)$ has the structure of a bocs. Now note that by \hyperref[anotherchara]{ \ref*{anotherchara}} $D(A) \cong Ae \otimes_{eAe} eA$ as $A$-bimodules and one can use this to get a bocs structure for $(A,D(A))$. \newline
Now we show that 2. implies 1.: \newline
Assume that $(A,D(A))$ is a bocs with comultiplication $\mu$ and counit $\epsilon$.
Note first that the comultiplication $\mu$ always has to be injective because in the identity $(\epsilon \otimes_A 1_W)\mu=c_r$ appearing the definition of a bocs, $c_r$ is an isomorphism. So there is a injection $\mu: D(A) \rightarrow D(A) \otimes_A D(A)$ which gives a surjection $D(\mu): D(D(A) \otimes_A D(A)) \rightarrow A$. Now using \hyperref[dualiso]{ \ref*{dualiso}} we see that $D(D(A) \otimes_A D(A)) \cong Hom_A(D(A),A)$ as $A$-bimodules. \newline 
Since $A$ is projective, $D(\mu)$ is split and $Hom_A(D(A),A) \cong A \oplus X$ for some $A$-right module $X$. By \hyperref[domdim1]{ \ref*{domdim1}}, this implies $Hom_A(D(A),A) \cong A$ and comparing dimensions, $D(\mu)$ and thus also $\mu$ have to be isomorphisms.
By \hyperref[anotherchara]{ \ref*{anotherchara}}, $A$ is gendo-symmetric.

\end{proof}

\begin{Corollary}
Let $A$ be a finite dimensional algebra. Then the following two conditions are equivalent: \newline
1. $A$ is gendo-symmetric. \newline
2. $\nu$ is a comonad.
\end{Corollary}
\begin{proof}
In \cite{BreWis} 18.28. it is proven that an $A$-bimodule $W$ is a bocs iff the functor $(-) \otimes_A W$ is a comonad. Appyling this with $W=D(A)$ and using the previous theorem, the corollary follows. 
\end{proof}
\begin{remark}
Theorem \hyperref[mainresult]{ \ref*{mainresult}} also shows that the comultiplication of the bocs $(A,D(A))$ is always an $A$-bimodule isomorphism for a gendo-symmetric algebra $A$. In \cite{FanKoe}, section 2.2., it is noted that such an isomorphism is unique up to multiples of invertible central elements in $A$. Thus the comultiplication of the bocs is also unique in that sense. 
\end{remark}

The following proposition gives an application:
\begin{proposition}
Let $A$ and $B$ be gendo-symmetric $K$-algebras. Then $A \otimes_K B$ is again a gendo-symmetric $K$-algebra. 
In particular, let $F$ be a field extension of $K$ and $A$ a gendo-symmetric $K$-algebra. Then $A \otimes_K F$ is again gendo-symmetric.
\end{proposition}
\begin{proof}
Let $A$ and $B$ two gendo-symmetric algebras. Then $\mathcal{B}_1=(A,D(A))$ and $\mathcal{B}_2=(B,D(A))$ are bocses. By example 3 of \hyperref[probasis]{ \ref*{probasis}} also the tensor product of $\mathcal{B}_1$ and $\mathcal{B}_2$ are bocses, it is the bocs $\mathcal{C}=(A \otimes_K B, D(A) \otimes_K D(B))$. Recall the well known formula $(D(A) \otimes_K D(B)) \cong D(A \otimes_K B)$, which can be found as exercise 12. of chapter II. in \cite{SkoYam}. Using this isomorphism one can find a bocs structure on $(A \otimes_K B, D(A \otimes_K B))$ using the bocs structure on $\mathcal{C}$. Thus by our bocs-theoretic characterization of gendo-symmetric algebras, also $A \otimes_K B$ is gendo-symmetric.
The second part follows since every field is a symmetric and thus gendo-symmetric algebra. 
\end{proof}

Let $A^{e}:=A^{op} \otimes_K A$ denote the enveloping algebra of a given algebra $A$.
The following proposition can be found in \cite{BreWis}, 17.8.
\begin{Proposition}
Let $(A,W)$ be a bocs and $c \in W$ with $\mu(c)=\sum\limits_{i=1}^{n}{c_{1,i} \otimes c_{2,i}}$. \newline
1. $Hom_A(W,A)$ has a ring structure with unit $\epsilon$ and product $*^{r}$, given as follows
for $f,g \in Hom_A(W,A)$: \newline
$f *^{r} g = g(f \otimes_A id_W) \mu$. \newline There is a ring anti-morphism $\zeta: A \rightarrow Hom_A(W,A)$, given by $\zeta(a)=\epsilon(a(-))$.  \newline
2. $Hom_{A^{e}}(W,A)$ has a ring structure with unit $\epsilon$ and multiplication $*$ given as follows for $f,g \in Hom_{A^{e}}(W,A)$: \newline
$f*g(c)= \sum\limits_{i=1}^{n}{f(c_{1,i})g(c_{2,i})}$.
\end{Proposition}

We now describe the ring structures on $Hom_{A^{e}}(D(A),A)$ and $Hom_A(D(A),A)$.
\begin{Proposition}
Let $A$ be gendo-symmetric. \newline
1. $\zeta$, as defined in the previous proposition, is a ring anti-isomorphism $\zeta: A \rightarrow  Hom_A(D(A),A)$. \newline
2. With the ring structure on $Hom_{A^{e}}(D(A),A)$ as defined in the previous proposition, $Hom_{A^{e}}(D(A),A)$ is isomorphic to the center $Z(A)$ of $A$.
\end{Proposition}
\begin{proof}
We use the isomorphism of $A$-bimodule $D(A) \cong Ae \otimes_{eAe} eA$. \newline
1. Since $A$ and $Hom_A(D(A),A)$ have the same $K$-dimension, the only thing left to show is that $\zeta$ is injective. So assume that $\zeta(a)=\epsilon(a(-))=0$, for some $a \in A$.
This is equivalent to $\epsilon(ax)=0$ for every $x= ce \otimes ed \in Ae \otimes eA$.
Now $\epsilon(a(ce \otimes ed))=\epsilon(ace \otimes ed)=aced$. Thus, since $c,d$ were arbitrary, $aAeA$=0. This means that $a$ is in the left annihilator $L(AeA)$ of the two-sided ideal $AeA$. But $L(AeA)=0$, since $Hom_A(A/AeA,A)=0$, by \hyperref[APT]{ \ref*{APT}} and thus $a=0$. Therefore $\zeta$ is injective. \newline
2. Define $\psi: Hom_{A^{e}}(D(A),A) \rightarrow Z(eAe)$ by $\psi(f)=f(e \otimes e)$, for \newline $f \in Hom_{A^{e}}(D(A),A)$. First, we show that this is well-defined, that is $f(e \otimes e)$ is really in the center of $Z(eAe)$. Let $x \in eAe$. Then $x f(e \otimes e)=f(xe \otimes e) = f(e \otimes ex)=f(e \otimes e)x$ and therefore $f(e \otimes e) \in Z(eAe)$. Clearly, $\psi$ is $K$-linear. Now we show that the map is injective: Assume $\psi(f)=0$, which is equivalent to $f(e \otimes e)=0$.
Then for any $a, b \in A: f(ae \otimes eb)=0$, and thus $f=0$. \newline
Now we show that $\psi$ is surjective. Let $z \in Z(eAe)$ be given. Then define a map $f_z \in Hom_{A^{e}}(D(A),A)$ by $f_z( ae \otimes eb )=zaeb$. Then, since $z$ is in the center of $eAe$, $f$ is $A$-bilinear and obviously $\psi(f_z)=f_z(e \otimes e)=ze=z$.
$\psi$ also preserves the unit and multiplication: \newline
$\psi(\epsilon)=\epsilon(e \otimes e)=e^2=e$ and for two given $f,g \in Hom_{A^{e}}(D(A),A)$: 
$\phi(f *g)= \newline (f*g)(e \otimes e)= (f*g)(e \otimes e) = f(e \otimes e) g(e \otimes e)$, by the definition of $*$.
To finish the proof, we use the result from \cite{FanKoe}, Lemma 2.2., that the map $\phi: Z(A) \rightarrow Z(eAe)$, $\phi(z)=eze$ is a ring isomorphism in case $A$ is gendo-symmetric. 
\end{proof}
\section{Description of the module category of the bocs $(A,D(A))$ for a gendo-symmetric algebra}
\setcounter{subsection}{1}
\setcounter{Definition}{0}
Let $A$ be a gendo-symmetric algebra. In this section we describe the module category of the bocs $\mathcal{B}=(A,D(A))$ as a $K$-linear category. We will use the $A$-bimodule isomorphism $Ae \otimes_{eAe} eA \cong D(A)$ often without mentioning. Let $M$ be an arbitrary $A$-module. Define for a given $M$ the map $I_M : M \rightarrow Hom_A(D(A),M)$ by $I_M(m)=u_m$ for any $m \in M$, where $u_m: D(A) \rightarrow M$ is the map $u_m( ae \otimes eb)=maeb$ for any $a,b \in A$.
Before we get into explicit calculation, let us recall how $*$ is defined in this special case.
Let $f \in Hom_{\mathcal{B}}(L,M)$ and $g \in Hom_{\mathcal{B}}(M,N)$, then for $l \in L$ and $a,b \in A:$
$(g*f)(l)(ae \otimes eb)=g(f(l)(ae \otimes e))(e \otimes eb)$.

\begin{Proposition}
\label{IM}
1. $I_M$ is well defined. \newline
2. $I_M$ is injective, iff $M$ has dominant dimension larger or equal 1. \newline
3. $I_M$ is bijective, iff $M$ has dominant dimension larger or equal 2.
\end{Proposition}
\begin{proof}
1. We have to show two things:
First, $u_m$ is $A$-linear for any $m \in M$: \newline $u_m((ae \otimes eb)c)=u_m(ae \otimes ebc)=m aebc=(maeb)c=u_m(ae \otimes eb)c.$ Second, $I_M$ is also $A$-linear: $I_M(mc)(ae \otimes eb)=u_{mc}(ae \otimes eb)=mcaeb=u_m(cae \otimes eb)= (u_m c)(ae \otimes eb)=(I_M(m)c)(ae \otimes eb)$. \newline
2. $I_M$ is injective iff ($m=0 \Leftrightarrow u_m=0$). Now $u_m=0$ is equivalent to $maeb=0$ for any $a,b \in A$. This is equivalent to the condition that the two-sided ideal $AeA$ annihilates $m$. Thus there is a nonzero $m$ with $u_m=0$ iff $Hom_A(A/AeA,M) \neq 0$ iff $M$ has dominant dimension zero by \hyperref[APT]{ \ref*{APT}}. \newline 
3. By \hyperref[domdim2]{ \ref*{domdim2}} $M$ has dominant dimension larger or equal two iff $M \cong \nu^{-1}(M)$. \newline Thus 3. follows by 2. since an injective map between modules of the same dimension is a bijective map. 
\end{proof}

\begin{Lemma}
\label{isonu}
For any module $M$, there is an isomorphism \newline
$Hom_A(\mu,M) \psi: Hom_A(D(A),Hom_A(D(A),M)) \rightarrow Hom(D(A),M) $ and thus \newline $\nu^{-1}(M) \cong \nu^{-2}(M)$.
It follows that every module of the form $\nu^{-1}(M)$ has dominant dimension at least two.
\end{Lemma}
\begin{proof}
The result follows, since $\psi$ is the canonical isomorphism \newline $\psi: Hom_A(D(A),Hom_A(D(A),M)) \rightarrow Hom_A(D(A) \otimes_A D(A),M) $ and since $\mu$ is an isomorphism also $Hom_A(\mu,M)$ is an isomorphism.
That $\nu^{-1}(M)$ has dominant dimension at least two, follows now from \hyperref[domdim2]{ \ref*{domdim2}}. 
\end{proof}

We define a functor $\phi: mod-A \rightarrow mod-\mathcal{B}$ by $\phi(M)=M$ and $\phi(f)=I_N f$ for an $A$-homomorphism $f :M \rightarrow N$. $\phi$ is obviously $K$-linear. The next result shows that it really is a functor and calculates its kernel on objects. 
\begin{theorem}
\label{mainresult2}
1. $\phi$ is a $K$-linear functor. \newline
2. $\phi(M)=0$ iff the two-sided ideal $AeA$ annihilates $M$, that is $M$ is a an $A/AeA$-module. All modules $M$ that are annihilated by $AeA$ have dominant dimension zero. \newline
3. By restricting $\phi$ to $Dom_2$, one gets an equivalence of $K$-linear categories $Dom_2 \rightarrow Dom_2^{\mathcal{B}}$, where $Dom_2^{\mathcal{B}}$ denotes the full subcategory of $mod-\mathcal{B}$ having objects all modules of dominant dimension at least 2. \newline
4. Any module $A$-module $M$ is isomorphic to $\nu^{-1}(M)$ in $\mathcal{B}$-mod and thus $\mathcal{B}$-mod is equivalent to $Dom_2$ as $K$-linear categories, which is equivalent to the module category mod-$eAe$.
\end{theorem} 
\begin{proof}
1. It was noted above that $\phi$ is $K$-linear. We have to show $\phi(id_M)= Hom(\epsilon,M) \zeta$, where $\zeta : M \rightarrow Hom_A(A,M)$ is the canconical isomorphism, and $\phi(g \circ f)= I_N(g) * I_M(f)$, where $f:L \rightarrow M$ and $g: M \rightarrow N$ are $A$-module homomorphisms.
To show the first equality $\phi(id_M)= Hom(\epsilon,M) \zeta$, just note that $Hom(\epsilon,M) \zeta(m)(ae \otimes eb)=l_m(\epsilon(ae \otimes eb))=maeb=I_M(m)(ae \otimes eb)$, where $l_m : A \rightarrow M$ is left multiplication by $m$. \newline
Next we show the above equality $\phi(g \circ f)= I_N(g) * I_M(f)$: \newline 
Let $l \in L$ and $a,b \in A$.
First, we calculate $\phi(g \circ f)(l)(ae \otimes eb)=g(f(l))aeb.$ \newline
Second, $I_N(g) * I_M(f)(l)(ae \otimes eb)=I_N(g)(I_M(f)(l)(ae \otimes e))(e \otimes eb)= \newline I_N(g)(u_{f(l)}(ae \otimes e))(e \otimes eb)=I_N(g)(f(l)(ae))(e \otimes eb)=g(f(l))aeb.$ \newline Thus $\phi(g \circ f)= I_N(g) * I_M(f)$ is shown. \newline
2. A module $M$ is zero in the $K$-category mod-$\mathcal{B}$ iff its endomorphism ring $End_{\mathcal{B}}(M)$ is zero iff the identity of $End_{\mathcal{B}}(M)$ is zero. Thus $M$ is zero in mod-$\mathcal{B}$ iff $I_M(m)=0$ for every $m \in M$. But $I_M(m)=0$ iff $mAeA=0$ and so $\phi(M)=0$ iff $MAeA=0$. To see that such an $M$ must have dominant dimension zero, note that $AeA$ annihilates no element of $M$ iff $M$ has dominant dimension larger or equal 1 by \hyperref[APT]{ \ref*{APT}}. \newline
3. Restricting $\phi$ to $Dom_2$, $\phi$ is obviously still dense by the definition of $Dom_2^{\mathcal{B}}$. Now recall that by the previous proposition a module $M$ has dominant dimension at least two iff $I_M$ is an isomorphism. Let now $h \in Hom_{\mathcal{B}}(M,N)$ be given with $M,N \in Dom_2^{\mathcal{B}}$. Then $\phi( I_N^{-1}h)=I_N(I_N^{-1}h))=h$ and $\phi$ is full. Assume $\phi(h)=I_N h=0$, then $h=0$, since $I_N$ is an isomorphism, and so $\phi$ is faithful. \newline
4. Define $f \in Hom_{\mathcal{B}}(M,\nu^{-1}(M))$ as $f=  (Hom_A(\mu,M) \psi)^{-1} I_M$ and \newline $g \in Hom_{\mathcal{B}}(\nu^{-1}(M),M)$ as $g=id_{\nu^{-1}(M)}$. 
We show that $f*g=I_{\nu^{-1}}(M)$ and $g*f=I_M$, which by 1. are the identities of  $Hom_{\mathcal{B}}(\nu^{-1}(M),\nu^{-1}(M))$ and $Hom_{\mathcal{B}}(M,M)$. This shows that any module $M$ is isomorphic to $\nu^{-1}(M)$ in $\mathcal{B}$-mod. \newline
Let $m \in M$ and $a,b \in A$. \newline
Then $(g*f)(m)(ae \otimes eb)=g(f(m)(ae \otimes e))(e \otimes eb)= \newline ((Hom_A(\mu,M) \psi)^{-1} I_M(m))((ae \otimes e))(e \otimes eb))=maeb=I_M(m)(ae \otimes eb)$, where we used that $g$ is the identity on $\nu^{-1}(M)$.
Next we show that $f*g=I_{\nu^{-1}}(M)$:
Let $l \in \nu^{-1}(M)=Hom_A(D(A),M)$. \newline First, note that by definition $I_{\nu^{-1}}(M)(l)(ae \otimes eb)(a'e \otimes eb')=(l aeb)(a' e \otimes eb')=l(aeba'e \otimes eb')$. 
Next $(f*g)(l)(ae \otimes eb)(a'e \otimes eb')=f(g(l)(ae \otimes eb)(a'e \otimes eb')= \newline f(l(ae \otimes eb))(a'e \otimes eb')=(Hom_A(\mu,M) \psi)^{-1} I_M(l(ae \otimes eb)(a'e \otimes eb')=l(ae \otimes eba'eb')=l(aeba'e \otimes eb')$, where we used in the last step that we tensor over $eAe$. \newline
Now we use \hyperref[isonu]{ \ref*{isonu}}, to show that every module of the form $\nu^{-1}(M)$ has dominant dimension at least two.
Since every module $M$ is isomorphic to $\nu^{-1}(M)$, $\mathcal{B}-mod$ is isomorphic to $Dom_2^{\mathcal{B}}$, which is isomorphic to $Dom_2$ by 3.
Now recall that there is an equivalence of categories mod-$eAe$ $\rightarrow Dom_2$ (this is a special case of \cite{APT} Lemma 3.1.). Combining all those equivalences, we get that $\mathcal{B}-mod$ is equivalent to the module category mod-$eAe$.
\end{proof}

\begin{Corollary}
In case an $A$-module $M$ has dominant dimension larger or equal $2$, the map \newline $Hom_A(M,I_M) : End_A(M) \rightarrow End_{\mathcal{B}}(M)$ is a $K$-algebra isomorphism. In particular $A \cong End_A(A) \cong End_{\mathcal{B}}(A)$, since $A$ has dominant dimension at least 2.
\end{Corollary}
\begin{proof}
This follows since $I_M$ is an isomorphism, in case $M$ has dominant dimension at least two by \hyperref[IM]{ \ref*{IM}} 3. 

\end{proof}

\begin{example}
Let $n \geq 2$ and $A:=K[x]/(x^n)$ and $J$ the Jacobson radical of $A$. Let \newline $M:= A \oplus \bigoplus_{k=1}^{n-1}{J^k}$ and $B:=End_A(M)$. Then $B$ is the Auslander algebra of $A$ and $B$ has $n$ simple modules. The idempotent $e$ is in this case primitiv and corresponds to the unique indecomposable projective-injective module $Hom_A(M,A)$. By the previous theorem, the kernel of $\phi$ is isomorphic to the module category $mod-(A/AeA)$. Here $A/AeA$ is isomorphic to the preprojective algebra of type $A_{n-1}$ by \cite{DR} chapter 7. \newline
We describe the bocs module category $\mathcal{B}$-mod of $(B,D(B))$ for $n=2$ explicitly. In this case $B$ is isomorphic to the Nakayama algebra with Kupisch series $[2,3]$. Then $B$ has five indecomposable modules. Let $e_0$ be the primitive idempotent corresponding to the indecomposable projective module with dimension two and $e_1$ the primitive idempotent corresponding to the indecomosable projective module with dimension three. Then $e_1A$ is the unique minimal faithful indecomposable projective-injective module. Let $S_i$ denote the simple $B$-modules. The only indecomposable module annihilated by $Be_1B$ is $S_0$, which is therefore isomorphic to zero in the bocs module category. The two indecomposable projective modules $P_0=e_0B$ and $P_1=e_1B$ have dominant dimension at least two and thus are not isomorphic. The only indecomposable module of dominant dimension 1 is $S_1$ and the only indecomposable module of dominant dimension zero, which is not isomorphic to zero in $\mathcal{B}$-mod, is $D(Be_0)$. Now let $X=S_1$ or $X=D(Be_0)$, then $\nu^{-1}(X)=Hom_B(D(B),X) \cong e_0B$. Thus in $\mathcal{B}$-mod $S_1 \cong e_0B \cong D(Be_0)$ and $e_1B$ are up to isomorphism the unique indecomposable objects. 
\end{example}

\ \newline
\noindent Ren\'{e} Marczinzik, Institute of algebra and number theory, University of Stuttgart, Pfaffenwaldring 57, 70569 Stuttgart, Germany \newline
E-Mail address: marczire@mathematik.uni-stuttgart.de

\end{document}